\documentclass[12pt, twoside, leqno]{article}

\usepackage{amsmath,amsthm}
\usepackage{amssymb}

\usepackage{enumerate}

\usepackage{graphicx}

\newtheorem{thm}{Theorem}[section]

\newtheorem{lem}[thm]{Lemma}

\theoremstyle{definition}
\newtheorem{defin}[thm]{Definition}

\newtheorem{exa}[thm]{Example}

\numberwithin{equation}{section}

\begin{document}


\baselineskip=17pt


\title{Discrete Fractional Solutions of a Physical Differential Equation via $ \nabla $-DFC Operator}

\author{Okkes Ozturk\\
Department of Mathematics\\ 
Bitlis Eren University\\
13000 Bitlis, Turkey\\
E-mail: oozturk27@gmail.com}

\date{}

\maketitle


\renewcommand{\thefootnote}{}

\footnote{2010 \emph{Mathematics Subject Classification}: Primary 26A33; Secondary 34A08, 39A70.}

\footnote{\emph{Key words and phrases}: fractional calculus, discrete fractional calculus (DFC), nabla discrete fractional operator, Leibniz rule, radial Schr{\"o}dinger equation.}

\renewcommand{\thefootnote}{\arabic{footnote}}
\setcounter{footnote}{0}


\begin{abstract}
Discrete mathematics, the study of finite structures, is one of the fastest growing areas in mathematics and optimization. Discrete fractional calculus (DFC) theory that is an important subject of the fractional calculus includes the difference of fractional order. In present paper, we mention the radial Schr{\"o}dinger equation which is a physical and singular differential equation. And, we can obtain the particular solutions of this equation by applying nabla ($ \nabla $) discrete fractional operator. This operator gives successful results for the singular equations, and solutions have fractional forms including discrete shift operator $ E $.
\end{abstract}

\section{Introduction}
Fractional calculus that deals with derivatives and integrals of arbitrary orders is a popular field of mathematics, and their applications appear in applied mathematics, physics, chemistry, engineering, and so on \cite{Baleanu2010, Baleanu2012, Miller, Oldham, Podlubny, Samko, Yilmazer}. There is a similarity between properties of fractional calculus and discrete fractional calculus (DFC) that includes difference of fractional order \cite{Diaz, Kuttner}. There are many examples of studies on DFC and Schr{\"o}dinger equation that is the main subject of this paper. A very simple perturbative numerical algorithm was developed \cite{Adam} for the solution of the radial Schr{\"o}dinger equation. The behaviour of the solution of the Schr{\"o}dinger radial equation was studied \cite{Hajj}. Gonzales et al. \cite{Gonzales} used integral equation method for the continuous spectrum radial Schr{\"o}dinger equation. Tselios and Simos \cite{Tselios} introduced new symplectic-schemes for the numerical solution of the radial Shr{\"o}dinger equation, and developed symplectic integrators for Hamiltonian systems. A simple exact analytical solution of the radial Schr{\"o}dinger equation for the Kratzer potential within the framework of the asymptotic iteration method (AIM) was presented \cite{Bayrak}. Ambrosetti and Ruiz \cite{Ambrosetti} proved the existence of radial solutions concentrating on spheres of nonlinear Schr{\"o}dinger equations with vanishing potentials. Anastassiou \cite{Anastassiou} developed the delta fractional calculus on time scales, and then produced related integral inequalities of types: Poincar?, Sobolev, Opial, Ostrowski and Hilbert-Pachpatte, and also introduced inequalities' applications on the time scales. Ferreira and Torres \cite{Ferreira} proved new results for the right fractional $ h $ sum, and mentioned the effectiveness of the obtained results in solving fractional discrete Euler-Lagrange equations. Oscillation of fractional nonlinear difference equations  was studied \cite{Marian}. Abdeljawad \cite{Abdeljawad} introduced two types of Caputo fractional differences, and investigated the relation between Riemann and Caputo fractional differences, and also provided the delta and nabla discrete Mittag-Leffler functions by solving Caputo type linear fractional difference equations. Ortigueira et al. \cite{Ortigueira} presented a derivative based discrete-time signal processing, and studied both nabla and delta derivatives, and also generalised including the fractional case. Wu and Baleanu \cite{Wu} studied on the analytical aspects, and the variational iteration method was extended in a new way to solve an initial value problem. The existence of solutions for antiperiodic boundary value problem and the Ulam stability for nonlinear fractional difference equations was mentioned \cite{ChenF}. Mohan and Deekshitulu \cite{MohanDeekshitulu} presented some important properties of N-transform, which is the Laplace transform for the nabla derivative on the time scale of integers, and obtained the N-transform of nabla fractional sums and differences and then applied this transform to solve some nabla fractional difference equations with initial value problems. Jonnalagadda \cite{Jonnalagadda} discussed the dependence of solutions of nabla fractional difference equations on the initial conditions and then obtained a fractional variation of constants formula for nabla fractional difference equations involving Caputo type fractional differences. A monotonicity result for discrete fractional difference operators was obtained \cite{Dahal}. Chen and Tang \cite{ChenY} expressed the differences between a class of fractional difference equations, and the integer-order ones, and indicated that under the same boundary conditions, the problem of the fractional order is nonresonant, while the integer-order one is resonant. Then, they analyzed the discrete fractional boundary value problem in detail, and obtained the uniqueness and multiplicity of the solutions for the discrete fractional boundary value problem. Mohan \cite{Mohan} mentioned the continuous dependence of solutions on the initial conditions for nabla fractional difference equations, and also obtained the linear variation of parameters formula for nabla fractional difference equations involving Riemann-Liouville type fractional differences. Dassios and Baleanu \cite{Dassios} formed a link between the solutions of an initial value problem of a linear singular system of fractional nabla difference equations. Atici and Uyanik \cite{Atici2015} studied two new monotonicity concepts for a nonnegative or nonpositive valued function defined on a discrete domain, and gave examples to illustrate connections between these new monotonicity concepts and the traditional ones.\\
Nabla discrete fractional operator of DFC is an important operator for the singular differential equation. Discrete fractional solutions of these equations can be obtained by means of this operator. In this paper, we applied this operator to the radial Schr{\"o}dinger equation given by potential $ V(r)=a/r^{2}-b/r+cr^{\rho} $ for different values of $ \rho $.

\section{Preliminaries}
\begin{defin} Riemann-Liouville fractional differentiation and fractional integration are defined by, respectively \cite{Yilmazer},
	\begin{equation}
	_aD_{t}^{\nu}f(t)=[f(t)]_{\nu}=\frac{1}{\Gamma(n-\nu)}\frac{d^{n}}{dt^{n}}\int_a^t\frac{f(s)}{(t-s)^{\nu+1-n}}ds, \label{A}
	\end{equation}
	\[ (n-1\leq\nu<n,n\in\mathbf{N}), \]
	and,
	\begin{equation}
	_aD_{t}^{-\nu}f(t)=[f(t)]_{-\nu}=\frac{1}{\Gamma(\nu)}\int_a^t\frac{f(s)}{(t-s)^{1-\nu}}ds\quad(t>a,\nu>0). \label{B}
	\end{equation}
\end{defin}

\begin{defin} Consider
	\begin{equation}
	t^{\overline{n}}=t(t+1)(t+2)...(t+n-1)\quad(n\in\mathbf{N}),
	\end{equation}
	and $ t^{\overline{0}}=1 $, where $ t^{\overline{n}} $ is the rising factorial power \cite{Graham}, or the ascending factorial \cite{Boros}. The rising factorial power can be defined by means of the Pochhammer symbol \cite{Spanier}.
	Let $ t\in\mathbf{R}\backslash\{...,-2,-1,0\} $ and $ \nu\in\mathbf{R} $ Then \textquotedblleft$ t $ to the $ \alpha $ rising\textquotedblright is given by
	\begin{equation}
	t^{\overline{\nu}}=\frac{\Gamma(t+\nu)}{\Gamma(t)},
	\end{equation}
	where $ 0^{\overline{\nu}}=0 $. Then, we can write equality as
	\begin{equation}
	\nabla(t^{\overline{\nu}})=\nu{t^{\overline{\nu-1}}},
	\end{equation}
	where $ \nabla{U(t)}=U(t)-U(t-1) $ \cite{Atici2009}.
\end{defin}

\begin{defin} Let $ a\in\mathbf{R} $ and $ \nu\in\mathbf{R}^{+} $ such that $ 0<n-1\leq\nu<n (n\in\mathbf{Z}) $.
	The $ \nu $-th order fractional sum of $ U $ is given by
	\begin{equation}
	\nabla_{a}^{-\nu}U(t)=\sum_{s=a}^{t}\frac{[t-\varphi(s)]^{\overline{\nu-1}}}{\Gamma(\nu)}U(s),
	\end{equation}
	where $ t\in\mathbf{N}_{a}=\{a,a+1,a+2,...\},\varphi(t)=t-1 $ is backward jump operator of the time scale calculus.
	The $ \nu $-th order fractional difference of $ U $ is given by
	\begin{equation}
	\nabla_{a}^{\nu}U(t)=\nabla^{n}\nabla_{a}^{-(n-\nu)}U(t)=\nabla^{n}\sum_{s=a}^{t}\frac{[t-\varphi(s)]^{\overline{n-\nu-1}}}{\Gamma(n-\nu)}U(s),
	\end{equation}
	where $ U:\mathbf{N}_{a}\rightarrow\mathbf{R} $ \cite{Acar}.
\end{defin}

\begin{defin} $ E $ shift operator is defined by
	\begin{equation}
	E^{n}U(t)=U(t+n),
	\end{equation}
	where $ n\in\mathbf{N} $ \cite{Kelley}.
\end{defin}

\begin{thm} Let $ \nu,\upsilon>0 $ and $ b,c $ are scalars. Then,
	\begin{equation}
	\nabla^{-\nu}\nabla^{-\upsilon}U(t)=\nabla^{-(\nu+\upsilon)}U(t)=\nabla^{-\upsilon}\nabla^{-\nu}U(t),
	\end{equation}
	\begin{equation}
	\nabla^{\nu}[bU(t)+cY(t)]=b\nabla^{\nu}U(t)+c\nabla^{\nu}Y(t),
	\end{equation}
	\begin{equation}
	\nabla\nabla^{-\nu}U(t)=\nabla^{-(\nu-1)}U(t),
	\end{equation}
	\begin{equation}
	\nabla^{-\nu}\nabla{U(t)}=\nabla^{1-\nu}U(t)-{t+\nu-2 \choose t-1}U(0),
	\end{equation}
	where $ U,Y:\mathbf{N}_{0}\rightarrow\mathbf{R} $ \rm\cite{Yilmazer2016}.
\end{thm}

\begin{lem}[Power Rule]
	\begin{equation}
	\nabla_{a}^{-\nu}(t-a+1)^{\overline{\upsilon}}=\frac{\Gamma(\upsilon+1)}{\Gamma(\nu+\upsilon+1)}(t-a+1)^{\overline{\nu+\upsilon}}\quad(\forall{t}\in\mathbf{N}_{a}),
	\end{equation}
	where $ \nu,\upsilon\in\mathbf{R} (\nu>0) $ \rm\cite{Acar}.
\end{lem}

\begin{lem}
	Let $ U $ is defined on $ \mathbf{N}_{a} $. The following equality holds \rm\cite{Acar}:
	\begin{equation}
	\nabla_{a+1}^{-\nu}\nabla{U(t)}=\nabla\nabla_{a}^{-\nu}U(t)-\frac{(t-a+1)^{\overline{\nu-1}}}{\Gamma(\nu)}U(a)\quad(\nu>0).
	\end{equation}
\end{lem}

\begin{lem}[Leibniz Rule]
	Let $ U(t) $ and $ Y(t) $ are defined on $ \mathbf{N}_{0} $. The $ \nu $-th order fractional difference of the product $ UY $ is defined by \rm\cite{Mohan2013}
	\begin{equation}
	\nabla_{0}^{\nu}(UY)(t)=\sum_{n=0}^{t}{\nu \choose n}[\nabla_{0}^{\nu-n}U(t-n)][\nabla^{n}Y(t)],\quad(\nu>0,t\in\mathbf{Z}^{+}),
	\end{equation}
	where $ {\nu \choose n}=\frac{\Gamma(\nu+1)}{\Gamma(\nu+1-n)n!} $.
\end{lem}

\begin{lem}
	Let $ U(t) $ is analytic and single-valued. The following equality holds \rm\cite{Yilmazer}:
	\begin{equation}
	[U_{\nu}(t)]_{\upsilon}=U_{\nu+\upsilon}(t)=[U_{\upsilon}(t)]_{\nu}\quad(\nu,\upsilon\in\mathbf{R},t\in\mathbf{N},U_{\nu}(t)\neq0,U_{\upsilon}(t)\neq0,
	\end{equation}
	where $ U_{\nu}=d^{\nu}U/dt^{\nu} $.
\end{lem}

\section{Main Results}
The radial Schr{\"o}dinger equation for the potential $ V(r)=a/r^{2}-b/r+cr^{\rho} $ is given by \cite{Aygun}
\begin{equation}
R_{2}(r)+\frac{2m}{\hbar^{2}}\Big[\epsilon-\frac{a}{r^{2}}+\frac{b}{r}-cr^{\rho}-\frac{\ell(\ell+1)\hbar^{2}}{2mr^{2}}\Big]R(r)=0,
\end{equation}
where $ a,b $ and $ c $ are positive constants. If we get equalities as
\begin{equation}
-\alpha^{2}=\frac{2m\epsilon}{\hbar^{2}},\quad\beta=\frac{2mb}{\hbar^{2}},\quad\gamma=\frac{2mc}{\hbar^{2}},\quad\delta=\frac{2ma}{\hbar^{2}}+\ell(\ell+1),
\end{equation}
then, we obtain
\begin{equation}
r^{2}R_{2}(r)-(\alpha^{2}r^{2}-\beta{r}+\gamma{r^{\rho+2}}+\delta)R(r)=0. \label{C}
\end{equation}
\begin{thm}\label{thm3.1}
	Let $ R\in\{R:0\neq\mid R_{\nu} \mid<\infty,\nu\in\mathbf{R}\} $ and $ \rho=0 $ in Eq. (\ref{C}). So,
	\begin{equation}
	r^{2}R_{2}(r)-\Big[(\alpha^{2}+\gamma)r^{2}-\beta{r}+\delta\Big]R(r)=0, \label{D}
	\end{equation}
	has particular solutions as follows:
	\[ R^{I}=Ae^{\eta{r}}r^{(\frac{1+\tau}{2})}\Big[e^{-2\eta{r}}r^{-(1+\tau+\mathrm{a})}\Big]_{-(1+\mathrm{a}E^{-1})}, \]
	\[ R^{II}=Be^{-\eta{r}}r^{(\frac{1+\tau}{2})}\Big[e^{2\eta{r}}r^{-(1+\tau+\mathrm{b})}\Big]_{-(1+\mathrm{b}E^{-1})}, \]
	\[ R^{III}=Ce^{\eta{r}}r^{(\frac{1-\tau}{2})}\Big[e^{-2\eta{r}}r^{-(1-\tau+\mathrm{c})}\Big]_{-(1+\mathrm{c}E^{-1})}, \]
	\[ R^{IV}=De^{-\eta{r}}r^{(\frac{1-\tau}{2})}\Big[e^{2\eta{r}}r^{-(1-\tau+\mathrm{d})}\Big]_{-(1+\mathrm{d}E^{-1})}, \]
	where $ A,B,C $ and $ D $ are constants.
\end{thm}
\begin{proof}
	At first, we suppose that $ R(r)=r^{\lambda}U(r) $, and then, we have
	\begin{equation}
	rU_{2}(r)+2\lambda{U_{1}(r)}+\Big[-(\alpha^{2}+\gamma)r+\beta+\Big(\lambda(\lambda-1)-\delta\Big)r^{-1}\Big]U(r)=0.
	\end{equation}
	If we get $ \lambda(\lambda-1)-\delta=0 $, so, we obtain $ \lambda=\frac{1\pm\tau}{2} $, where $ \tau=\sqrt{1+4\delta} $.\\
	Let $ \lambda=\frac{1+\tau}{2} $. Therefore, we write equation as
	\begin{equation}
	rU_{2}(r)+(1+\tau){U_{1}(r)}+\Big[-(\alpha^{2}+\gamma)r+\beta\Big]U(r)=0.
	\end{equation}
	After, consider $ U(r)=e^{\sigma{r}}Y(r) $, and so,
	\begin{equation}
	rY_{2}(r)+(2\sigma{r}+1+\tau)Y_{1}(r)+\Big[\Big(\sigma^{2}-(\alpha^{2}+\gamma)\Big)r+\sigma(1+\tau)+\beta\Big]Y(r)=0. \label{E}
	\end{equation}
	We choose $ \sigma $ such that $ \sigma^{2}-(\alpha^{2}+\gamma)=0 $ in Eq. (\ref{E}), that is, $ \sigma=\pm\eta $, where $ \eta=\alpha^{2}+\gamma $. For $ \sigma=\eta $, we obtain
	\begin{equation}
	rY_{2}(r)+(2\eta{r}+1+\tau)Y_{1}(r)+\Big[\eta(1+\tau)+\beta\Big]Y(r)=0. \label{F}
	\end{equation}
	Now, by applying $ \nabla^{\nu} $ discrete fractional operator to the both sides of Eq. (\ref{F}), we have
	\begin{equation}
	rY_{2+\nu}(r)+(\nu{E}+2\eta{r}+1+\tau)Y_{1+\nu}(r)+\Big[\eta(2\nu{E}+1+\tau)+\beta\Big]Y_{\nu}(r)=0, \label{G}
	\end{equation}
	where $ E $ is shift operator. Here, we suppose that $ \eta(2\nu{E}+1+\tau)+\beta=0 $, that is, $ \nu=\mathrm{a}E^{-1} $, where $ \mathrm{a}=-\Big(\frac{\beta+\eta(1+\tau)}{2\eta}\Big) $. So, we obtain a first-order homogeneous linear ordinary differential equation as follows:
	\begin{equation}
	\varphi_{1}(r)+\Big[2\eta+(1+\tau+\mathrm{a})r^{-1}\Big]\varphi(r)=0, \label{H}
	\end{equation}
	where $ [Y(r)]_{(1+\mathrm{a}E^{-1})}=\varphi(r) $ and $ Y(r)=[\varphi(r)]_{-(1+\mathrm{a}E^{-1})} $. We have the solution of Eq. (\ref{H}) as
	\begin{equation}
	\varphi(r)=Ae^{-2\eta{r}}r^{-(1+\tau+\mathrm{a})},
	\end{equation}
	and finally, we find a discrete fractional solution of the Eq. (\ref{D}) as follows:
	\begin{equation}
	R(r)=Ae^{\eta{r}}r^{(\frac{1+\tau}{2})}\Big[e^{-2\eta{r}}r^{-(1+\tau+\mathrm{a})}\Big]_{-(1+\mathrm{a}E^{-1})}.
	\end{equation}
	Similarly, we obtain other discrete fractional solutions as follows:
	\begin{equation}
	R(r)=Be^{-\eta{r}}r^{(\frac{1+\tau}{2})}\Big[e^{2\eta{r}}r^{-(1+\tau+\mathrm{b})}\Big]_{-(1+\mathrm{b}E^{-1})}\quad\Bigg(\mathrm{b}=\Big(\frac{\beta-\eta(1+\tau)}{2\eta}\Big)\Bigg),
	\end{equation}
	\begin{equation}
	R(r)=Ce^{\eta{r}}r^{(\frac{1-\tau}{2})}\Big[e^{-2\eta{r}}r^{-(1-\tau+\mathrm{c})}\Big]_{-(1+\mathrm{c}E^{-1})}\quad\Bigg(\mathrm{c}=-\Big(\frac{\beta+\eta(1-\tau)}{2\eta}\Big)\Bigg),
	\end{equation}
	\begin{equation}
	R(r)=De^{-\eta{r}}r^{(\frac{1-\tau}{2})}\Big[e^{2\eta{r}}r^{-(1-\tau+\mathrm{d})}\Big]_{-(1+\mathrm{d}E^{-1})}\quad\Bigg(\mathrm{d}=\Big(\frac{\beta-\eta(1-\tau)}{2\eta}\Big)\Bigg).
	\end{equation}
\end{proof}

We can write the following theorems with the implementation of similar steps:
\begin{thm}\label{thm3.2}
	Let $ R\in\{R:0\neq\mid R_{\nu} \mid<\infty,\nu\in\mathbf{R}\} $ and $ \rho=-1 $ in Eq. (\ref{C}). So,
	\[ r^{2}R_{2}(r)-\Big[\alpha^{2}r^{2}+(\gamma-\beta)r+\delta\Big]R(r)=0, \]
	has particular solutions as follows:
	\[ R^{I}=Ae^{\alpha{r}}r^{(\frac{1+\tau}{2})}\Big[e^{-2\alpha{r}}r^{-(1+\tau+\mathrm{a})}\Big]_{-(1+\mathrm{a}E^{-1})}\quad\Bigg(\mathrm{a}=-\Big(\frac{\beta-\gamma+\alpha(1+\tau)}{2\alpha}\Big)\Bigg), \]
	\[ R^{II}=Be^{-\alpha{r}}r^{(\frac{1+\tau}{2})}\Big[e^{2\alpha{r}}r^{-(1+\tau+\mathrm{b})}\Big]_{-(1+\mathrm{b}E^{-1})}\quad\Bigg(\mathrm{b}=\Big(\frac{\beta-\gamma-\alpha(1+\tau)}{2\alpha}\Big)\Bigg), \]
	\[ R^{III}=Ce^{\alpha{r}}r^{(\frac{1-\tau}{2})}\Big[e^{-2\alpha{r}}r^{-(1-\tau+\mathrm{c})}\Big]_{-(1+\mathrm{c}E^{-1})}\quad\Bigg(\mathrm{c}=-\Big(\frac{\beta-\gamma+\alpha(1-\tau)}{2\alpha}\Big)\Bigg), \]
	\[ R^{IV}=De^{-\alpha{r}}r^{(\frac{1-\tau}{2})}\Big[e^{2\alpha{r}}r^{-(1-\tau+\mathrm{d})}\Big]_{-(1+\mathrm{d}E^{-1})}\quad\Bigg(\mathrm{d}=\Big(\frac{\beta-\gamma-\alpha(1-\tau)}{2\alpha}\Big)\Bigg), \]
	where $ \lambda=\frac{1\pm\tau}{2},\tau=\sqrt{1+4\delta},\sigma=\pm\alpha $ and $ A,B,C,D $ are constants.
\end{thm}
\begin{thm}\label{thm3.3}
	Let $ R\in\{R:0\neq\mid R_{\nu} \mid<\infty,\nu\in\mathbf{R}\} $ and $ \rho=-2 $ in Eq. (\ref{C}). So,
	\[ r^{2}R_{2}(r)-\Big[\alpha^{2}r^{2}-\beta{r}+(\gamma+\delta)\Big]R(r)=0, \]
	has particular solutions as follows:
	\[ R^{I}=Ae^{\alpha{r}}r^{(\frac{1+\tau}{2})}\Big[e^{-2\alpha{r}}r^{-(1+\tau+\mathrm{a})}\Big]_{-(1+\mathrm{a}E^{-1})}\quad\Bigg(\mathrm{a}=-\Big(\frac{\beta+\alpha(1+\tau)}{2\alpha}\Big)\Bigg), \]
	\[ R^{II}=Be^{-\alpha{r}}r^{(\frac{1+\tau}{2})}\Big[e^{2\alpha{r}}r^{-(1+\tau+\mathrm{b})}\Big]_{-(1+\mathrm{b}E^{-1})}\quad\Bigg(\mathrm{b}=\Big(\frac{\beta-\alpha(1+\tau)}{2\alpha}\Big)\Bigg), \]
	\[ R^{III}=Ce^{\alpha{r}}r^{(\frac{1-\tau}{2})}\Big[e^{-2\alpha{r}}r^{-(1-\tau+\mathrm{c})}\Big]_{-(1+\mathrm{c}E^{-1})}\quad\Bigg(\mathrm{c}=-\Big(\frac{\beta+\alpha(1-\tau)}{2\alpha}\Big)\Bigg), \]
	\[ R^{IV}=De^{-\alpha{r}}r^{(\frac{1-\tau}{2})}\Big[e^{2\alpha{r}}r^{-(1-\tau+\mathrm{d})}\Big]_{-(1+\mathrm{d}E^{-1})}\quad\Bigg(\mathrm{d}=\Big(\frac{\beta-\alpha(1-\tau)}{2\alpha}\Big)\Bigg), \]
	where $ \lambda=\frac{1\pm\tau}{2},\tau=\sqrt{1+4(\gamma+\delta)},\sigma=\pm\alpha $ and $ A,B,C,D $ are constants.
\end{thm}

\subsection{Some Examples}
\begin{exa}
	Consider a radial Schr{\"o}dinger equation as
	\begin{equation}
	r^{2}R_{2}-(5r^{2}-2r+2)R=0,\label{ex1}
	\end{equation}
	where
	\[ \tau=3,\quad\eta=5,\quad a=-\frac{11}{5},\quad b=-\frac{9}{5},\quad c=\frac{4}{5},\quad d=\frac{6}{5}.  \]
	Then, we obtain the particular solutions of Eq. (\ref{ex1}) by means of Theorem \ref{thm3.1} as follows:
	\[ R(r)=Ae^{5r}r^{-1}(e^{-10r}r^{6/5})_{-9/5}=Ce^{5r}r^{2}{}_{1}\textbf{F}_{1}\Big[\frac{11}{5},4,-10r\Big], \]
	and
	\[ R(r)=Be^{-5r}r^{-1}(e^{10r}r^{4/5})_{-11/5}=De^{-5r}r^{2}{}_{1}\textbf{F}_{1}\Big[\frac{9}{5},4,10r\Big], \]
	where $ A,B $ are arbitrary constants and  $ C=0.183634A, D=0.155231B $ and $ {}_{1}\textbf{F}_{1} $ is hypergeometric function.
\end{exa}

\begin{exa}
	We have
	\[ \tau=3,\quad\eta=1,\quad a=-1,\quad b=-3,\quad c=2,\quad d=0,  \]
	for a radial equation as
	\begin{equation}
	r^{2}R_{2}-(r^{2}+2r+2)R=0.\label{ex2}
	\end{equation}
	According to Theorem \ref{thm3.2}, we have two particular solutions of Eq. (\ref{ex2}) as
	\[ e^{r}r^{2}(e^{-2r}r^{-3})_{0}\Rightarrow e^{r}r^{-1}(e^{-2r})_{-3}\Rightarrow R(r)=A\frac{e^{-r}}{r}, \]
	and
	\[ e^{-r}r^{2}(e^{2r}r^{-1})_{2}\Rightarrow e^{-r}r^{-1}(e^{2r}r^{2})_{-1}\Rightarrow R(r)=B\frac{e^{r}(1-2r+2r^{2})}{r}, \]
	where $ A,B $ are arbitrary constants.
\end{exa}

\begin{exa}
	Let
	\begin{equation}
	r^{2}R_{2}-(r^{2}-2r+6)R=0.\label{ex3}
	\end{equation}
	So, we obtain
	\[ \tau=5,\quad\eta=1,\quad a=-4,\quad b=-2,\quad c=1,\quad d=3.  \]
	And, by means of Theorem \ref{thm3.3}, we find two particular solutions of Eq. (\ref{ex3}) as
	\[ e^{r}r^{3}(e^{-2r}r^{-2})_{3}\Rightarrow e^{r}r^{-2}(e^{-2r}r^{3})_{-2}\Rightarrow R(r)=A\frac{e^{-r}(6+9r+6r^{2}+2r^{3})}{r^{2}}, \]
	and
	\[ e^{-r}r^{3}(e^{2r}r^{-4})_{1}\Rightarrow e^{-r}r^{-2}(e^{2r}r)_{-4}\Rightarrow R(r)=B\frac{e^{r}(r-2)}{r^{2}}, \]
	where $ A,B $ are arbitrary constants.
\end{exa}
\section*{Conclusion}
We obtained the explicit solutions of the radial Schr{\"o}dinger equation for the potential $ V(r)=a/r^{2}-b/r+cr^{\rho} $ by means of Leibniz rule (\textit{nabla discrete fractional operator}) in DFC. So, we introduced the discrete fractional solutions for $ \rho=0,-1,-2 $. After determining variables, the useful results can be obtained by using equalities (\ref{A}) or (\ref{B}).


\end{document}